\documentclass{article}
\usepackage{geometry}
\usepackage{amsfonts}
\usepackage{array}
\usepackage{amsmath,mathtools}
\usepackage{amsthm}
\usepackage[utf8]{inputenc}
\usepackage[english]{babel}
\usepackage{amssymb}
\usepackage{tikz}
\usepackage{caption}
\usepackage{cite}
\usepackage{appendix}

\theoremstyle{definition}
\newtheorem{defn}{Definition}[section]

\newtheorem{prop}{Proposition}[section]
\newtheorem{thm}{Theorem}[section]

\newtheorem{rem}{Remark}[section]
\newtheorem{ex}{Example}[section]
 
\newtheorem{lemma}[thm]{Lemma}

\newtheorem{que}{Question}

\usepackage[utf8]{inputenc}

\title{Positroids are 3-colorable}
\author{Lamar Chidiac \and Winfried Hochst\"{a}ttler}
\date{March 2024}

\begin{document}

\maketitle
\begin{center}
	{\footnotesize
		Fakult\"at f\"ur Mathematik und Informatik, \\
		FernUniversit\"at in Hagen, Germany,\\
		\text{\{lamar.chidiac,winfried.hochstaettler\}@fernuni-hagen.de}
		\\\ \\
	}
\end{center}

\begin{abstract}
	We show that every positroid of rank $r \geq 2$ has a good
	coline. Using the definition of the chromatic number of oriented
	matroid introduced by J.\ Ne\v{s}et\v{r}il, R.\ Nickel, and W.~Hochst\"{a}ttler, this shows that every orientation of a positroid of rank at least $2$
	is $3$-colorable.
\end{abstract}

\section{Introduction}

A classical algebraic way to analyze the chromatic number of a graph
is to study Nowhere-Zero(NZ)-tensions, which we will call NZ-coflows,
as they form a concept dual to NZ-flows.  Flows and coflows
immediately generalize to regular oriented matroids and
Hochst\"{a}ttler \cite{HadwigerforHyperplanes} presented a
generalization of Hadwiger's famous conjecture that any graph which is
not $k$-colorable must have a $K_{k+1}$-minor to regular oriented
matroids. Goddyn and Hochst\"attler \cite{Boergerband} proved that
this conjecture is equivalent to Tutte's 4-Flow conjecture for $k=4$,
Tutte's 5-flow conjecture for $k=5$ and Hadwiger's conjecture for
$k\ge 6$.

A matroid is regular if and only if it has no minor isomorphic to $U_{2,4}$ (the
four point line), the Fano plane or the dual Fano plane. A first idea to generalize the theory of
NZ-coflows to general oriented matroids would be, to require that
summing the coflow-values along any cycle, taking the signs of the
elements into account, evaluates to zero. Alas, one has to cope with
the fact that using this definition no orientation of $U_{2,4}$ admits
a non-trivial coflow, in particular, there is no NZ-coflow. Therefore, one needs to take another approach to define the NZ-coflows for general oriented matroids. 
The coflow space may dually be defined as the linear hull of signed characteristic vectors of cocircuits. Using this and ideas of Hochst\"attler and
Ne\v{s}et\v{r}il \cite{honese}, Hochstättler and Nickel
\cite{theflowlatticeofOM, onthechromaticnumber} initiated a chromatic
theory of oriented matroids. Lacking the total unimodularity of the
matrix defining the cocircuit space, they considered the integer
lattice (chain group) generated by the signed characteristic vectors
of the cocircuits instead. It seems possible that the generalization of
Hadwiger's Conjecture to regular oriented matroids might be true even
for general orientable matroids.

While the case $k=2$ still is trivial for general oriented matroids,
already the case $k=3$, which was proven by Hadwiger in the graphic
(and regular) case, is open in general. The graphic case is easily
proved observing the fact that a simple graph without $K_4$-minor
always has a vertex of degree at most $2$. Therefore Goddyn,
Hochst\"{a}ttler and Neudauer~\cite{goddynetal}, introduced the class
of generalized series parallel ($GSP$) oriented matroids, which is an
$M(K_4)$-free class, requiring that the coflow-lattice of every minor
contains a vector with at most two non-zero entries which must be $+1$
or $-1$. This class is easily seen to be 3-colorable in the sense of
Hochst\"{a}ttler and Nickel, if the matroid is loopless.  In general,
the class of $M(K_4)$-free matroids is not well
understood~\cite{geelenblog}. If every such orientable matroid could
be proven to be $GSP$, then Hadwiger's conjecture would hold for 
oriented matroids in the case $k=3$. A large class of orientable
matroids without $M(K_4)$-minor is formed by the gammoids.
Unfortunately, they are only slightly better understood.

One possibility to prove membership in the class $GSP$ is to find two
compatible cocircuits in the oriented matroid with symmetric
difference of cardinality at most two.  The existence of such a pair
is guaranteed for any orientation if there exists a coline with more
simple than multiple copoints.  Call such a coline a positive coline. Here, a copoint  $H$ (a hyperplane in the
underlying matroid) is simple with respect to the coline $L$ if
$|H \setminus L|=1$. Goddyn et
al.\ showed in \cite{goddynetal} that every simple bicircular matroid of rank at least $2$
has a positive coline. Bicircular matroids are transversal matroids, and the smallest class under minors that contains transversal matroids is the class of gammoids. Therefore they conjectured in the same work that the same holds for every
simple gammoid of rank at least $2$. Albrecht and Hochst\"{a}ttler
\cite{Immanuelandwinfried, immanueldiss} later proved this to be true for rank $3$.
However, Guzman-Pro and Hochst\"{a}ttler recently proved in \cite{santiago} that not every simple gammoid of rank at least $2$ has a positive coline, disproving by that the conjecture of Goddyn et al. \cite{goddynetal}. They also proved in the same work that the existence of a good coline (a coline with at least two simple copoints) is enough to prove membership to the GSP class, and that such a coline exists in every simple cobicircular matroid of rank at least $2$.

Every loopless oriented matroid of rank $3$ not containing $M(K_4)$ is 3-colorable \cite{onthechromaticnumber}. In addition, until now bicircular matroids of rank $\geq 2$ \cite{goddynetal} and simple lattice path matroids of rank $\geq 2$\cite{Immanuelandwinfried1} were  shown to be $3$-colorable by proving the existence of a positive coline in each one of them. In this paper, we prove the following theorem by proving the existence of a good coline in any simple positroid of rank at least $2$.
\begin{thm}
	 Simple positroids of rank at least $2$ are $3$-colorable.
	\end{thm}

We  assume basic familiarity with graph theory, matroid
theory and oriented matroids. The standard references are~\cite{oxleybook,ombibel}. 

\section{Preliminaries}

In this work we consider matroids to be pairs $M=(E,\mathcal{B})$
where $E=[n]=\{1,2,...n\}$ and present them as their set of bases. 
For simplicity sake, we omit curly brackets and commas when writing subsets. So rather than $\mathcal{B}=\{\{1,2\},\{1,3\}\}$, we write $\mathcal{B}=\{12,13\}$.
A \emph{path} in a digraph $D=(V,A)$ is a non-empty and non-repeating sequence $p_{1}p_{2}\dots p_{n}$ of vertices $p_i \in V$ such that for each $1 \leq i <n$, $(p_{i},p_{i+1}) \in A$. By convention, we shall denote $p_n$ by $p_{-1}$. Furthermore, the set of vertices traversed by a path $p$ shall be denoted by $p=\{p_{1},p_{2}, \dots,p_{n}\}$.

As mentioned in the introduction, gammoids are the main motivation behind this work, however they are not very well understood. For a deeper understanding of gammoids, we refer the reader to \cite{immanueldiss}. In this section we define gammoids and positroids, and show that positroids are actually gammoids. For that we are going to need the following definition.

\begin{defn}
	Let $D=(V,A)$ be a digraph, and $X,Y \subseteq V$. A \emph{routing} from $X$ to $Y$ in $D$ is a family of paths $R$, such that: 
	\begin{itemize}
		\item for each $x$ in $X$ there is some $p \in R$ with $p_{1}=x$ 
		\item for all $p \in R$ the end vertex $p_{-1} \in Y$, and 
		\item for all $p,q \in R$, either $p=q$ or $p\cap q=\emptyset$
	\end{itemize}   
	A routing $R$ is called \emph{linking} from $X$ to $Y$, if it is a routing onto $Y$, i.e.\ whenever $Y =\{p_{-1} \:| \: p \in R \}$. We write $X \rightrightarrows Y$ whenever there is a routing from $X$ to $Y$. Furthermore, we say that $v$ is linked to $w$ in a directed graph, whenever there exists a directed path from $v$ to $w$ or from $w$ to $v$.
\end{defn}

\begin{defn}[Gammoids]
	Let $D=(V,A)$ be a digraph, $E\subseteq V$, and $T \subseteq V$. The \emph{gammoid} represented by $(D,T,E)$ is defined to be the matroid $M(D,T,E)=(E,\mathcal{I})$ where 
	\begin{equation*}
		\mathcal{I}=\{X \subseteq E \; | \; \text{there is a routing} \; X \rightrightarrows T \; \text{in D} \}.
	\end{equation*}
\end{defn}  

Positroids were originally defined in \cite{postnikov} as the column
sets coming from nonzero maximal minors of a real matrix such that all
maximal minors are nonnegative. Thus positroids are matroids
representable over the reals and therefore orientable. However, in this
paper we use an equivalent definition using
\reflectbox{L}-graphs.

\begin{defn}[\reflectbox{L}-diagram]
	A \reflectbox{L}-\emph{diagram} is a lattice path with a finite collection of boxes lying above it, arranged in left-justified rows, with the row lengths in non-increasing order, where some of the boxes are filled with dots in such a way that the \reflectbox{L}-\emph{property} is satisfied: If a box has a dot above it in the same column and to its left in the same row, then this box should contain a dot. Figure \ref{L-diagram} shows an example of a  \reflectbox{L}-diagram.
\end{defn} 

\begin{defn}[\reflectbox{L}-graph]
	A \reflectbox{L}-graph is constructed from a \reflectbox{L}-diagram as follows.
	\begin{itemize}
		\item Add a vertex in the middle of each edge of the lattice path and label them from $1$ to $n$ starting at the Northeast corner of the path and ending at the Southwest corner. We refer to these vertices as \emph{sources} and \emph{sinks}, depending on whether they lie on a horizontal or vertical step. Note that we do not think of this as a normal lattice path, where we start at the origin and then move north or east. So the first step of the lattice path is the edge where the vertex with the label $1$ is, and the last step is the edge where the vertex with the label $n$ is.
		
		The vertices of the \reflectbox{L}-graph are the  dots inside the boxes in the \reflectbox{L}-diagram and the vertices added in this step.
		
		\item Add edges directed to the right between any two consecutive vertices that lie in the 
		same row.
		\item Add edges directed upwards between any two consecutive  vertices that lie in the same column.
	\end{itemize} 
	A \reflectbox{L}-\emph{graph} $G=(V,A)$ is the graph obtained from a \reflectbox{L}-diagram following the above steps (see Figure \ref{le-graph}), where $V$ is the set vertices formed by the internal vertices (dots inside boxes) and the external vertices (vertices placed on the lattice path), and $A$ is the set of directed edges added to the \reflectbox{L}-diagram previously. Edges between external vertices are not in $A$. The \reflectbox{L}-property implies that $G$ is always planar.
\end{defn}

The following provides an example of a \reflectbox{L}-diagram and a \reflectbox{L}-graph.

\begin{center}
	\begin{figure}[h]
		\begin{minipage}{.45\textwidth}
			\centering
			\begin{tikzpicture}[scale=0.9]
				\draw[help lines,line width=.5pt,step=1] (0,0) grid (2,1);
				\draw[help lines,line width=.5pt,step=1] (0,1) grid (3,3);
				\draw[line width=0.5mm] (4,3)--(3,3)--(3,1)--(2,1)--(2,0)--(0,0);
				\foreach \Point in {(1.5,2.5),(2.5,1.5),(1.5,1.5),(0.5,1.5),(1.5,0.5)}{
					\node at \Point {\textbullet};}
			\end{tikzpicture}
			\caption{A \reflectbox{L}-diagram $D$}
			\label{L-diagram}
		\end{minipage}\hspace{2cm}
		\begin{minipage}{.45\textwidth} 
			\centering
			\begin{tikzpicture}[scale=0.9]
				\draw[help lines,line width=.5pt,step=1] (0,0) grid (2,1);
				\draw[help lines,line width=.5pt,step=1] (0,1) grid (3,3);
				\draw[line width=0.5mm] (4,3)--(3,3)--(3,1)--(2,1)--(2,0)--(0,0);
				\draw[thick,->](1.5,2.5) -- (3,2.5);
				\draw[thick,->](2.5,1.5) -- (3,1.5);
				\draw[thick,->](1.5,1.5) -- (2.5,1.5);
				\draw[thick,->](0.5,1.5) -- (1.5,1.5);
				\draw[thick,->](1.5,0.5) -- (2,0.5);
				\draw[thick,->](1.5,1.5) -- (1.5,2.5);
				\draw[thick,->](1.5,0.5) -- (1.5,1.5);
				\draw[thick,->](1.5,0) -- (1.5,0.5);
				\draw[thick,->](0.5,0) -- (0.5,1.5);
				\draw[thick,->](2.5,1) -- (2.5,1.5);
				
				\draw(3.5,3) node[below] {1};
				\draw(3,2.5) node[right] {2};
				\draw(3,1.5) node[right] {3};
				\draw(2.5,1) node[below] {4};
				\draw(2,0.5) node[right] {5};
				\draw(1.5,0) node[below] {6};
				\draw(0.5,0) node[below] {7};
				
				\foreach \Point in {(3.5,3),(3,2.5),(3,1.5),(2.5,1),(2,0.5),(1.5,0),(0.5,0),(1.5,2.5),(2.5,1.5),(1.5,1.5),(0.5,1.5),(1.5,0.5)}{
					\node at \Point {\textbullet};}
			\end{tikzpicture}
			\caption{The corresponding \reflectbox{L}-graph obtained from $D$}
			\label{le-graph}
		\end{minipage}
	\end{figure}
\end{center}

According to \cite{postnikov}, the following definition of positroids is equivalent to all of the definitions of positroids found there. We recommend to the reader the work of Ardila et al. \cite{ardila} for a better understanding of positroids. 
\begin{defn}[Positroids]
	Let $\Gamma$ be a \reflectbox{L}-graph and let $B$ be the set of sinks, that is the set of external vertices labeling vertical edges of the lattice path, let $|B|=r$. We define the collection $\mathcal{I}$ of independent sets of the positroid $\mathcal{P}(\Gamma)$ to be all sets $I \in $ $[n] \choose k$ where $k \leq r$, such that there exists $k$ pairwise vertex-disjoint paths from $I$ to $B$ in $\Gamma$.  Note that if $x \in I \cap B$ then $x$ is the trivial path from $I$ to $B$ and therefore every element of $B$ is an independent set on its own.
	A matroid $M=([n],\mathcal{I})$ is a positroid if there is a \reflectbox{L}-graph $\Gamma$ such that $M \cong \mathcal{P}(\Gamma)$. The rank of the positroid $r(\mathcal{P})$ is $r$ where $r$ is the maximum size of an independent set in the positroid, which is equal to the number of sinks.
\end{defn}

As it is commonly done, we identify a matroid with its collection of basis. These are the independent sets of size $r$ where $r$ is the rank of the positroid. For example, the positroid built from the previous \reflectbox{L}-graph in Figure \ref{le-graph} is
\begin{equation*}
	\mathcal{P} = \{235,236,245,246,256,257,267,356,357,367,456,457,467\}.
\end{equation*}

It is easy to see that positroids are actually gammoids. For any independent set $I$ in a positroid $\mathcal{P}$, there is a routing between $I$ and $B$, where $B$ is the set of all sinks in $\mathcal{P}$. Therefore a positroid is a gammoid with representation $(\Gamma,B,[n])$ with $M(\Gamma,B,[n])=(E,\mathcal{I})$ where 
\begin{equation*}
	\mathcal{I}=\{X \subseteq E \; | \; \text{there is a routing} \; X \rightrightarrows B \; \text{in}\; D \}
\end{equation*}

\begin{rem}[Rank of a subset of a positroid]
	In order to understand the proof of our main theorem, it is crucial to understand how we can compute the rank of an arbitrary subset of a positroid $\mathcal{P}(\Gamma)$ using the \reflectbox{L}-graph $\Gamma$. We refer the reader to our latest work \cite{lamaralgorithm} where we explain in details the algorithm that computes the rank of an arbitrary subset of a positroid. As you already know, the rank $r(S)$ of a set $S$ is the maximum size of an independent subset in $S$. Therefore, $r(S)=k$ such that $k$ is the maximum number of pairwise vertex-disjoint paths from $S$ to $B$, where $B$ is the set of all sinks of $\Gamma$. We give here a quick way to calculate the rank. 
	
	Let $S$ be an arbitrary subset of a positroid $\mathcal{P}$ that contains $x$ sinks. We have that $r(S)= x+t$, such that $t$ is the maximum number of vertex-disjoint walks from $S\backslash B$ (the sources in $S$) to $B\backslash S$ (the sinks not in $S$).
\end{rem}

In this work we show that simple positroids are 3-colorable, by proving the existence of a good coline. For that, we recall the following definitions. 

\begin{defn}[Closure operator]
	Let $M$ be a matroid on the set $E$, and let $2^E$ be the collection of subsets of $E$. The \emph{closure operator} of $M$ is the function $cl$: $2^E \to 2^E$ given as follows: for $X \subseteq E$,
	\begin{equation*}
		cl(X) = \{e: r(X \cup e) = r(X)\}
	\end{equation*}
	The \emph{flats} or \emph{closed sets} of $M$ are the set image of the operator $cl$.
\end{defn}

\begin{defn}[Copoints and Colines]
	Let $M$ be a matroid. A \emph{copoint} of $M$ is a flat of codimension $1$, that is a hyperplane. A \emph{coline} is a flat of codimension 2. If $L$ is a coline of $M$, $H$ a copoint of $M$ and $L \subseteq H$, then we 
	say that $H$ is a \emph{copoint on} $L$. The copoint is \emph{simple} (with respect to $L$) if $|H \backslash L| = 1$, otherwise it is \emph{multiple}. A coline is \emph{good} if there are at least $2$ simple copoints on $L$. 
\end{defn}

\section{Loops, coloops and parallel elements in a \reflectbox{L}-graphs}

In the \reflectbox{L}-graph, a loop of the positroid is represented by a source that is not linked to anything. A coloop is represented by a sink that is not linked to anything. Figure \ref{loopcoloop} below presents the \reflectbox{L}-graph of a positroid containing loops and a coloop.

Since the set of all sinks in the \reflectbox{L}-graph is a base of the positroid, a pair of elements can be parallel only if both elements are sources or if one is a source and the other is a sink. Let $h_i$ and $h_j$ be two sources in the \reflectbox{L}-graph of a positroid $\mathcal{P}$ such that $h_i$ is to the right of $h_j$, and let $w$ be the first internal vertex above $h_i$.

\begin{center}
	\begin{figure}[h]
		\begin{minipage}{.45\textwidth}
			\centering
			\begin{tikzpicture}[scale=0.9]
				\draw[help lines,line width=.5pt,step=1] (0,0) grid (2,1);
				\draw[help lines,line width=.5pt,step=1] (0,1) grid (3,3);
				\draw[line width=0.5mm] (4,3)--(3,3)--(3,1)--(2,1)--(2,0)--(0,0);
				\draw[thick,->](1.5,2.5) -- (3,2.5);
				\draw[thick,->](1.5,1.5) -- (3,1.5);
				\draw[thick,->](0.5,1.5) -- (1.5,1.5);
				\draw[thick,->](1.5,1.5) -- (1.5,2.5);
				\draw[thick,->](1.5,0) -- (1.5,1.5);
				\draw[thick,->](0.5,0) -- (0.5,1.5);
				\draw[thick,->](0.5,1.5) -- (0.5,2.5);
				\draw[thick,->](0.5,2.5) -- (1.5,2.5);
				
				\draw(3.5,3) node[below] {1};
				\draw(3,2.5) node[right] {2};
				\draw(3,1.5) node[right] {3};
				\draw(2.5,1) node[below] {4};
				\draw(2,0.5) node[right] {5};
				\draw(1.5,0) node[below] {6};
				\draw(0.5,0) node[below] {7};
				
				\foreach \Point in {(3.5,3),(3,2.5),(3,1.5),(2.5,1),(2,0.5),(1.5,0),(0.5,0),(1.5,2.5),(1.5,1.5),(0.5,1.5),(0.5,2.5)}{
					\node at \Point {\textbullet};}
			\end{tikzpicture}
			\caption{The \reflectbox{L}-graph of a positroid in which the sources 1 and 4 are loops, and the sink 5 is a coloop.}
			\label{loopcoloop}
		\end{minipage}\hspace{2cm}
		\begin{minipage}{.45\textwidth} 
			\centering
			\begin{tikzpicture}[scale=0.9]
				\draw[help lines,line width=.5pt,step=1] (0,0) grid (2,1);
				\draw[help lines,line width=.5pt,step=1] (0,1) grid (3,3);
				\draw[line width=0.5mm] (4,3)--(3,3)--(3,1)--(2,1)--(2,0)--(0,0);
				\draw[thick,->](1.5,2.5) -- (3,2.5);
				\draw[thick,->](2.5,1.5) -- (3,1.5);
				\draw[thick,->](0.5,0.5) -- (1.5,0.5);
				\draw[thick,->](1.5,0.5) -- (2,0.5);
				\draw[thick,->](1.5,0.5) -- (1.5,2.5);
				\draw[thick,->](1.5,0) -- (1.5,0.5);
				\draw[thick,->](0.5,0) -- (0.5,0.5);
				\draw[thick,->](2.5,1) -- (2.5,1.5);
				
				\draw(3.5,3) node[below] {1};
				\draw(3,2.5) node[right] {2};
				\draw(3,1.5) node[right] {3};
				\draw(2.5,1) node[below] {4};
				\draw(2,0.5) node[right] {5};
				\draw(1.5,0) node[below] {6};
				\draw(0.5,0) node[below] {7};
				
				\foreach \Point in {(3.5,3),(3,2.5),(3,1.5),(2.5,1),(2,0.5),(1.5,0),(0.5,0),(1.5,2.5),(2.5,1.5),(0.5,0.5),(1.5,0.5)}{
					\node at \Point {\textbullet};}
			\end{tikzpicture}
			\caption{The \reflectbox{L}-graph of a positroid in which \{3,4\} and \{6,7\} are two parallel pairs.}
			\label{parallel}
		\end{minipage}
	\end{figure}
\end{center}

\textbullet \; $h_i$ and $h_j$ are parallel if and only if any directed path starting from $h_j$ must go through $w$.\\

\textbullet \; A source $h$ and a sink $v$ are parallel if and only if $h$ is not linked to any sink other than $v$ and $h$ is linked to $v$ . \\

Figure \ref{parallel} above presents a \reflectbox{L}-graph of a positroid with parallel elements. In the rest of this work, we consider only simple positroids, i.e.\ positroids with no loops and parallel elements. Therefore we assume that each source in the \reflectbox{L}-graph is linked to at least two sinks and no two sources are parallel.

\section{Connectivity of positroids}
For this section we recall the following proposition from \cite{oxleybook} about the matroid's connectivity.

\begin{prop}[Proposition 4.1.3 in \cite{oxleybook}]\label{connectedmatroid}
	A matroid $M$ is \emph{connected} if and only if, for every pair of
	distinct elements of $E(M)$, there is a circuit containing both.
\end{prop}

In order to characterize connectivity of positroids in \reflectbox{L}-graphs, we need the following definitions.

\begin{defn}
	Let $\mathcal{P}$ be a positroid and $D$ its \reflectbox{L}-diagram. A \emph{level} in $D$ is a collection of consecutive horizontal edges of the lattice path with the consecutive vertical edges that follows them.
\end{defn}
For example, in Figure \ref{parallel}, the edges of the lattice path labelled by $2,3$ and $4$ form a level and the edges of the lattice path labelled by $5,6$ and $7$ form another level.

\begin{defn}
	Let $\mathcal{P}$ be a positroid and $D$ be its \reflectbox{L}-diagram. An \emph{isolated block} in $D$ is a level or a collection of levels in which the sources of this block are only connected to the sinks in it.
	We can write this formally in the following way.

	\begin{equation*}
		B \text{ is an isolated block in }D \iff \begin{cases}
			\forall v \text{ a sink} \in B, \forall h \text{ a source} \notin B, h\text{ is not linked to }v,  \text{ and}\\
			\forall h \text{ a source} \in B, \forall v \text{ a sink} \notin B, h \text{ is not linked to }v.
		\end{cases}
	\end{equation*}
\end{defn}

The following figures are examples of \reflectbox{L}-diagrams with isolated blocks.

\begin{center}
	\begin{minipage}{.45\textwidth}
		\centering
		\begin{tikzpicture}[scale=0.8]
			\draw[help lines,line width=.4pt,step=1] (-1,-2) grid (0,3);
			\draw[help lines,line width=.4pt,step=1] (0,-1) grid (2,3);
			\draw[help lines,line width=.4pt,step=1] (2,2) grid (3,3);
			
			\draw[line width=0.5mm] ((3,3)--(3,1)--(2,1)--(2,-1)--(0,-1)--(0,-2)--(-1,-2);
			\draw[thick,->](2.5,2.5) -- (3,2.5);
			\draw[thick,->](2.5,1.5) -- (3,1.5);
			\draw[thick,->](1.5,0.5) -- (2,0.5);
			\draw[thick,->](1.5,-0.5) -- (1.5,0.5);
			\draw[thick,->](0.5,0.5) -- (1.5,0.5);
			\draw[thick,->](2.5,1) -- (2.5,1.5);
			\draw[thick,->](2.5,1.5) -- (2.5,2.5);
			\draw[thick,->](0.5,-0.5) -- (0.5,0.5);
			\draw[thick,->](0.5,-1) -- (0.5,-0.5);
			\draw[thick,->](0.5,-0.5) -- (1.5,-0.5);
			\draw[thick,->](1.5,-1) -- (1.5,-0.5);
			\draw[thick,->](1.5,-0.5) -- (2,-0.5);
			\draw[thick,->](-0.5,-2) -- (-0.5,-1.5);
			\draw[thick,->](-0.5,-1.5) -- (0,-1.5);
			\draw[thick,->](-0.5,-1.5) -- (-0.5,1.5);
			\draw[thick,->](-0.5,1.5) -- (2.5,1.5);
			
			\draw(3,2.5) node[right] {1};
			\draw(3,1.5) node[right] {2};
			\draw(2.5,1) node[below] {3};
			\draw(2,0.5) node[anchor=north west] {4};
			\draw(2,-0.5) node[right] {5};
			\draw(0,-1.5) node[anchor=north west] {8};
			\draw(1.5,-1) node[below] {6};
			\draw(0.5,-1) node[below] {7};
			\draw(-0.5,-2) node[below] {9};

			\foreach \Point in {(2,-0.5),(0.5,-0.5),(1.5,-0.5),(-0.5,-1.5),(-0.5,1.5),(-0.5,-2),(0.5,-1),(1.5,-1),(2,0.5),(2.5,2.5),(3,2.5),(3,1.5),(2.5,1),(2.5,1.5),(1.5,0.5),(0.5,0.5),(0,-1.5)}{
				\node at \Point {\textbullet};}
		\end{tikzpicture} 
		\captionof{figure}{A \reflectbox{L}-diagram with an isolated block formed by one level: edges labelled by 4,5,6 and 7. The other isolated block consists of two levels: edges labelled by 1,2,3,8 and 9}
	\end{minipage}	\hspace{1cm}
	\begin{minipage}{.45\textwidth}
		\centering
		\begin{tikzpicture}[scale=0.8]
			\draw[help lines,line width=.4pt,step=1] (-1,-2) grid (0,3);
			\draw[help lines,line width=.4pt,step=1] (0,-1) grid (1,3);
			\draw[help lines,line width=.4pt,step=1] (1,0) grid (2,3);
			\draw[help lines,line width=.4pt,step=1] (2,2) grid (3,3);
			
			\draw[line width=0.5mm] ((3,3)--(3,2)--(2,2)--(2,0)--(1,0)--(1,-1)--(0,-1)--(0,-2)--(-1,-2);
			\draw[thick,->](2.5,2.5) -- (3,2.5);
			\draw[thick,->](2.5,2) -- (2.5,2.5);
			\draw[thick,->](1.5,0.5) -- (2,0.5);
			\draw[thick,->](1.5,0) -- (1.5,0.5);
			\draw[thick,->](0.5,0.5) -- (1.5,0.5);
			\draw[thick,->](1.5,0.5) -- (1.5,1.5);
			\draw[thick,->](1.5,1.5) -- (2,1.5);
			\draw[thick,->](0.5,-0.5) -- (0.5,0.5);
			\draw[thick,->](0.5,-1) -- (0.5,-0.5);
			\draw[thick,->](0.5,-0.5) -- (1,-0.5);
			\draw[thick,->](-0.5,-2) -- (-0.5,-1.5);
			\draw[thick,->](-0.5,-1.5) -- (0,-1.5);
			\draw[thick,->](-0.5,-1.5) -- (-0.5,2.5);
			\draw[thick,->](-0.5,2.5) -- (2.5,2.5);
			
			\draw(3,2.5) node[right] {1};
			\draw(2.5,2) node[below] {2};
			\draw(2,1.5) node[anchor=north west] {3};
			\draw(2,0.5) node[right] {4};
			\draw(1.5,0) node[below] {5};
			\draw(1,-0.5) node[anchor=north west] {6};
			\draw(0.5,-1) node[below] {7};
			\draw(0,-1.5) node[anchor=north west] {8};
			\draw(-0.5,-2) node[below] {9};

			\foreach \Point in {(2,1.5),(1.5,0),(1.5,1.5),(0.5,-0.5),(-0.5,-1.5),(-0.5,2.5),(-0.5,-2),(0.5,-1),(1,-0.5),(2,0.5),(2.5,2.5),(3,2.5),(2.5,2),(1.5,0.5),(0.5,0.5),(0,-1.5)}{
				\node at \Point {\textbullet};}
		\end{tikzpicture} 
		\captionof{figure}{A \reflectbox{L}-diagram with two isolated blocks. Block 1: edges labelled by 3,4,5,6,7. Block 2: edges labelled by 1,2,8,9}
	\end{minipage}
\end{center}


We next prove that isolated blocks are strongly related to the connectivity of a positroid.
\begin{lemma}
	A positroid $\mathcal{P}$ is connected if and only if there exists no two disjoint isolated blocks in its \reflectbox{L}-diagram. Equivalently, a positroid $\mathcal{P}$ is not connected if and only if there exist two disjoint isolated blocks in its \reflectbox{L}-diagram $D$. 
\end{lemma}

\begin{proof}
	In the following we prove that a positroid is not connected if and only if there are disjoint isolated blocks. We start by the necessary condition. Let $B_1, \dots, B_k$ be the isolated blocks in $D$. These isolated blocks have disjoint set of elements and therefore form a partition of the elements of $P$ which can be written as the direct sum of these isolated blocks, $\mathcal{P}=B_{1} \oplus \dots \oplus B_{k}$. Thus, $P$ is not connected since clearly no circuit can contain two elements from two disjoint isolated blocks.
	
	For the sufficient condition, we now assume that $P$ is not connected. As being on a common circuit is an equivalence relation, we find a partition $P=P_{1} \oplus \dots \oplus P_k$ with $k \ge 2$. We now show that the $P_i$'s are the isolated blocks of $D$. Let $B(P_i)$ (resp. $H(P_i)$) be the set of sinks (resp. sources) of the isolated block $P_i$. Therefore $\forall$ $e \in P_i$ and $\forall$ $f \in P_j$ with $i \neq j$,  $e$ and $f$ do not belong to a common circuit. We now consider two cases. If $e \in B(P_i)$, then $\forall f \in H(P_j)$ ($i \neq j$), $f$ is not linked to $e$ because otherwise $e,f$ and all sinks linked to $f$ form a circuit, which contradicts the fact that $e$ and $f$ do not belong to a common circuit. Similarly, if $e \in H(P_i)$, then $\forall f \in B(P_j)$ ($i \neq j$), $e$ is not linked to $f$ because otherwise $e,f$ and all sinks linked to $e$ form a circuit. Thus $P_i$ is an isolated block in $D$, $\forall$  $1 \leq i \leq k$. 
	\end{proof}

\begin{rem}
	In \cite{Bonin2010}, J. Bonin proved that a lattice path matroid $L$ of rank $r(L)$ is connected if and only if it has a spanning circuit, that is a circuit of rank  $r(L)$. Unlike lattice path matroids, a positroid can be connected without having a spanning circuit. In the following, we present a  counterexample (Figure \ref{counterex}) showing a connected positroid and then prove the absence of a spanning circuit in it.
	
	\begin{center}
		\begin{figure}[h]
			\centering
			\begin{tikzpicture}[scale=1]
				
				\draw[help lines,line width=.5pt,step=1] (0,0) grid (2,1);
				\draw[help lines,line width=.5pt,step=1] (0,0) grid (1,-1);
				\draw[help lines,line width=.5pt,step=1] (0,1) grid (3,3);
				
				\draw[line width=0.5mm] (3,3)--(3,1)--(2,1)--(2,0)--(1,0)--(1,-1)--(0,-1);
				\draw[thick,->](2.5,2.5) -- (3,2.5);
				\draw[thick,->](1.5,2.5) -- (2.5,2.5);
				\draw[thick,->](2.5,1.5) -- (3,1.5);
				\draw[thick,->](0.5,2.5) -- (1.5,2.5);
				\draw[thick,->](1.5,0.5) -- (2,0.5);
				\draw[thick,->](0.5,-0.5) -- (1,-0.5);
				\draw[thick,->](1.5,0) -- (1.5,0.5);
				\draw[thick,->](0.5,-1) -- (0.5,2.5);
				\draw[thick,->](2.5,1) -- (2.5,1.5);
				\draw[thick,->](1.5,0.5) -- (1.5,2.5);
				\draw[thick,->](2.5,1.5) -- (2.5,2.5);
				
				\draw(3,2.5) node[right] {1};
				\draw(3,1.5) node[right] {2};
				\draw(2.5,1) node[below] {3};
				\draw(2,0.5) node[right] {4};
				\draw(1.5,0) node[below] {5};
				\draw(1,-0.5) node[right] {6};
				\draw(0.5,-1) node[below] {7};
				
				\foreach \Point in {(0.5,-1),(1,-0.5),(3,2.5),(3,1.5),(2.5,1),(2,0.5),(1.5,0),(2.5,2.5),(2.5,1.5),(0.5,2.5),(1.5,0.5),(0.5,-0.5),(1.5,2.5)}{
					\node at \Point {\textbullet};}
			\end{tikzpicture}
			\caption{A connected positroid with no spanning circuit.}
			\label{counterex}
		\end{figure}
	\end{center}
	
	The positroid in Figure \ref{counterex} is connected since there are no two disjoint isolated blocks in the \reflectbox{L}-diagram. Suppose this positroid has a spanning circuit $C$. It is easy to notice that $C$ must contain at least one of $6$ and $7$, since otherwise $C$ doesn't have full rank and therefore is not spanning. Furthermore, $C$ must contain $6$ and $7$ simultaneously, because if only one of them was in $C$, then its removal will prevent us from reaching full rank and therefore the rank of $C$ will decrease after its removal. Similarly, both $4$ and $5$ must be in $C$. Until now, $C$ contains the set $\{4,5,6,7\}$. However, this set is not independent, it is of rank $3$ and has $4$ elements. Therefore this set cannot be properly contained in any circuit, which contradicts our assumption that $C$ is a spanning circuit. 
\end{rem}

\section{Main theorem and proof}
Before we get into our main theorem and its proof, let us mention the following propositions, which explain why we assume the connectivity of positroids in the proof.

\begin{prop}[\cite{ardila} Proposition 3.5]\label{ardila}
	Positroids are closed under taking minors and duals.
\end{prop} 

\begin{prop} \label{prop}
	Let $M$ be a non-connected matroid such that each component has rank $\ge 2$. If one connected component of $M$ has a good coline then $M$ has a good coline.
\end{prop}

\begin{proof}
    Since $M$ is not connected, we may assume that $M$ can be written as the direct sum $M=M_{1}\oplus M_{2}$ of two matroids $M_1$ of rank $r_1$ and $M_2$
	of rank $r_2$ where $M_1$ is connected and $r_1 , r_2 \ge 2$. If $M_1$ has
	a good coline $L$, it is easy to see that $L \cup M_2$ is a good
	coline of $M$, because $L \cup M_2$ is a flat of rank $r-2$ and it is
	contained in the same number of simple and multiple copoints of
	$M_1$. In fact, if $C_1$ (resp.\ $C_2$) is a simple (resp.\ multiple)
	copoint on $L$, then $C_{1} \cup M_2$ (resp.\ $C_{2} \cup M_2$) is simple (resp.\
	multiple) copoint on $L \cup M_2$. Thus $M$ has a good coline.
\end{proof}


Now, we present and prove our main theorem. 

\begin{thm} \label{goodcoline}
	Every simple positroid of rank $r \geq 2$ has a good coline.
\end{thm}

\begin{proof}
	In a simple positroid of rank $1$, a coline is not defined and a simple positroid of rank $2$ has the empty set as a coline and thus every copoint is simple. For the rest of the proof we work with simple positroids of rank at least $3$. 
	Due to proposition \ref{prop}, it is enough to prove the theorem for connected positroids. Therefore, in the rest of the proof we assume connectivity of the positroids, i.e.\
	there exists no two disjoint isolated blocks (in particular, no coloops or loops). Moreover, since we only consider simple
	positroids, i.e.\ positroids with no
	loops and parallel elements, each source should be linked to at least two sinks, otherwise the source and the sink linked to it form a parallel pair. Thus, the first two steps of the lattice path in the \reflectbox{L}-diagram of the positroid have to be vertical. Recall that we consider the lattice path as starting from the Northeast and ending in the Southwest.
	 Let $V=\{v_{1},...,v_{r}\}$ be the set of sinks and $H=\{h_{1},...,h_{n-r}\}$ be the set of sources of a positroid $\mathcal{P}$ with $v_1 \le \ldots \le v_r$ and $h_1 \le \ldots \le h_{n-r}$. If a sink $v_j$ was followed by a source for some $j \in [r]$, we denote that source by $h^j$ .\\
	 
\textbf{Case 1:} There exists $j \in [r]$ such that $h^j$ is the only source linked to the sink $v_j$.

 Since $v_j$ and $h^j$ are not parallel, $h^j$ must be linked to some $v_i$ with $i<j$. Let $v_{a}^j$ be the highest sink $h^j$ is linked to. In other words, for any $k < a$, $h^j$ is not linked to $v_k$. We prove in this case that $L= cl(V \setminus \{v_{a}^j , v_{j}\})$ is a good coline. Note here that the set $L$ contains all sources that are not linked to neither $v_{a}^j$ nor $v_j$.
 It is clear that the set $H_1 =L \cup v_j$ is of rank $r-1$. To prove that $H_1$ is also closed we have to prove that adding any element to it will increase its rank. Clearly, $r(H_1 \cup v_{a}^j)=r$, and since $h^j$ is the source linked to both $v_{a}^j$ and $v_j$, we also have that $r(H_1 \cup h^j )=r$. All other sources that are not in $L$ are sources linked to $v_{a}^j$ and not to $v_j$ and therefore adding any of these sources to $H_1$ will increase its rank. Thus $H_1$ is closed and a simple copoint on $L$.
 Let $H_2 = L \cup h^j$. We have $r(H_2) = r-1$. Since $h^j$ is linked to both $v_{a}^j$ and $v_j$, it is clear that $r(H_2 \cup v_{a}^j)=r(H_2 \cup v_j)=r$. All other sources that are not in $L$ are sources linked to $v_{a}^j$ and not to $v_j$ and therefore adding any of these sources to $H_2$ will increase its rank. Thus $H_2$ is a simple copoint on $L$. This proves that $L$ is a good coline.
 
  \begin{ex}
 	Let us consider the following positroid as an example for case 1.
 \begin{center}
 		\begin{tikzpicture}[scale=0.8]
 			\draw[help lines,line width=.4pt,step=1] (-2,-1) grid (0,-3);
 			\draw[help lines,line width=.4pt,step=1] (-2,1) grid (2,-1);
 			\draw[help lines,line width=.4pt,step=1] (-2,1) grid (3,3);
 			
 			\draw[line width=0.5mm] (3,3)--(3,1)--(2,1)--(2,-1)--(1,-1)--(1,-2)--(0,-2)--(0,-3)--(-2,-3);
 			
 			\draw[thick,->](2.5,2.5) -- (3,2.5);
 			\draw[thick,->](2.5,1.5) -- (3,1.5);
 			\draw[thick,->](1.5,0.5) -- (2,0.5);
 			\draw[thick,->](1.5,-0.5) -- (1.5,0.5);
 			\draw[thick,->](0.5,0.5) -- (1.5,0.5);
 			\draw[thick,->](2.5,1) -- (2.5,1.5);
 			\draw[thick,->](2.5,1.5) -- (2.5,2.5);
 			\draw[thick,->](0.5,-0.5) -- (0.5,0.5);
 			\draw[thick,->](-1.5,2.5) -- (2.5,2.5);
 			\draw[thick,->](1.5,-1) -- (1.5,-0.5);
 			\draw[thick,->](1.5,-0.5) -- (2,-0.5);
 			\draw[thick,->](0.5,-2) -- (0.5,-1.5);
 			\draw[thick,->](0.5,-1.5) -- (1,-1.5);
 			\draw[thick,->](0.5,-1.5) -- (0.5,-0.5);
 			\draw[thick,->](0.5,-0.5) -- (1.5,-0.5);
 			\draw[thick,->](-0.5,-3) -- (-0.5,-2.5);
 			\draw[thick,->](-0.5,-2.5) -- (0,-2.5);
 			\draw[thick,->](-0.5,-2.5) -- (-0.5,-1.5);
 			\draw[thick,->](-0.5,-1.5) -- (0.5,-1.5);
 			\draw[thick,->](-1.5,-3) -- (-1.5,0.5);
 			\draw[thick,->](-1.5,0.5) -- (-1.5,2.5);
 			\draw[thick,->](-1.5,0.5) -- (0.5,0.5);
 			
 			\draw(3,2.5) node[right] {$v_{1}$};
 			\draw(3,1.5) node[right] {$v_{2}$};
 			\draw(2.5,1) node[below] {$h_{1}$};
 			\draw(2,0.5) node[anchor=north west] {$v_{3}$};
 			\draw(2,-0.5) node[right] {$v_{4}$};
 			\draw(1.5,-1) node[below] {$h_{2}$};
 			\draw(1,-1.5) node[anchor=north west] {$v_{5}$};
 			\draw(0.5,-2) node[below] {$h_{3}$};
 			\draw(0,-2.5) node[anchor=north west] {$v_{6}$};
 			\draw(-0.5,-3) node[below] {$h_{4}$};
 			\draw(-1.5,-3) node[below] {$h_{5}$};
 			
 			\foreach \Point in {(-0.5,-3),(0.5,-2),(0,-2.5),(1.5,-1),(2,0.5),(2.5,2.5),(3,2.5),(3,1.5),(2.5,1),(2.5,1.5),(1.5,0.5),(0.5,0.5),(-1.5,2.5),(-1.5,0.5),(1,-1.5),(2,-0.5),(-1.5,-3),(0.5,-0.5),(1.5,-0.5),(0.5,-1.5),(-0.5,-1.5),(-0.5,-2.5)}{
 				\node at \Point {\textbullet};}
 		\end{tikzpicture} 
 		\captionof{figure}{}
 		\label{case1}
\end{center}
This positroid has two sinks that are linked to only one source each. Namely $h_1$ is the only source linked to $v_2$ and $h_4$ is the only source linked to $v_6$. Since $v_1$ is the highest sink $h_1$ is linked to, the coline $L_1 =  cl(V \setminus \{v_1 , v_2\})$ is a good coline. It is easy to check that the copoints $H_1 = L_1 \cup v_2$ and $H_2 = L_1 \cup h_1$ are simple copoints on $L_1$. Similarly, since $v_3$ is the highest sink $h_4$ is linked to, the coline $L_2 = cl(V \setminus \{v_3 , v_6\})$ is a good coline. It is easy to check that the copoints $H_3 = L_2 \cup v_6$ and $H_4 = L_2 \cup h_4$ are simple copoints on $L_2$. 
  \end{ex}
 
\textbf{Case 2:} For all $j \in [r]$ there are at least two sources linked to the sink $v_j$.

Let $v_i$ and $v_{i+1}$ be the last two consecutive sinks (not separated by a source) of $\mathcal{P}$ which exist as $v_1$ and $v_2$ are two consecutive sinks. We consider here two subcases.
\begin{itemize}
	\item Subcase 2.1: There exist two consecutive sources (not separated by a sink) $h_{t}^j$ and $h_{t+1}^j$ such that $j > i$. 
	
	Let $h_{t}^j$ and $h_{t+1}^j$ be the first two consecutive sources such that $j > i$, and let $w$ be the internal vertex directly above $h_{t}^j$.
	Since $h_{t}^j$ is not the only source linked to $v_j$, $h_{t+1}^j$ must also be linked to $v_j$. Since $v_i$ and $v_{i+1}$ are the last two consecutive sinks and since $h_{t}^j$ and $h_{t+1}^j$ are the first two consecutive sources such that $j>i$, the steps of the lattice path between $v_{i+1}$ and $h_{t}^j$ (if they exist) are steps alternating between a source and a sink. Moreover, since every source should be linked to at least two sinks, $h^{i+2}$ must be linked to $v_{i+2}$ and a higher sink. If $h^{i+2}$ was not linked to $v_{i+1}$, any other source to the left of $h^{i+2}$ could not be linked to $v_{i+1}$ as well, due to the planarity of the \reflectbox{L}-graph. If this was the case, $h^{i+1}$ would be the only source linked to $v_{i+1}$ and we are now back to Case 1. Therefore, $h^{i+2}$ must be linked to $v_{i+1}$. Similarly, $h^{i+3}$ must be linked to $v_{i+2}$ etc., and $h_{t}^j$ must be linked to $v_{j-1}$. In other words, $h^{s}$ must be linked to $v_{s-1}$ for all $s$ such that $i < s \leq j$. We call this \emph{the connection property} between $v_{i+1}$ and $h_{t}^j$. It follows that $h_{t}^j$ is linked to $v_i$. Note that this still holds if there are no steps between $v_{i+1}$ and $h_{t}^j$. We also know that not all paths starting from $h_{t+1}^j$ go through $w$, otherwise $h_{t}^j$ and $h_{t+1}^j$ are two parallel elements. We now have two possibilities.
	\begin{itemize}
		\item $h_{t+1}^j$ is linked to $v_i$ through a path avoiding $w$.
		
		We prove in this case that $L=cl (V \setminus \{v_i , v_j\})$ is a good coline. The set $H_1 = L \cup v_j$ is of rank $r-1$. The source $h_{t}^j$ is linked to $v_j$ and $v_i$ and therefore every source that is linked to $v_j$ is also linked to $v_i$ through the internal vertex $w$. That is because every source that is linked to $v_j$ must go through $w$, and starting from $w$ we can find a path to $v_i$ (using the path from $h_{t}^j$ to $v_i$). We now have $r(H_{1} \cup h_{t}^j)=r$ and $r(H_{1} \cup v_i)=r$ and the addition of any other source to $H_1$ will increase its rank. Therefore $H_1$ is closed and a simple copoint on $L$. Similarly, $H_2= L \cup h_{t}^j$ is a simple copoint on $L$. It is easy to check that $r(H_2)=r-1$ and that $r(H_2 \cup h_{t+1}^j)=r$. Since $h_{t+1}^j$ is linked to $v_i$ through a path avoiding $w$, any other source linked to both $v_j$ and $v_i$ also have a path avoiding $w$ due to the planarity of the \reflectbox{L}-graph. Thus, adding any of these sources to $H_2$ will increase its rank. Note here that the set $L \cup v_i$ might not be closed since all sources between $v_{i+1}$ and $h_{t}^j$ (if they exist), are linked to $v_{i}$ and not to $v_j$. Figure \ref{case2.1.1} gives an example of this case.
		
		\item $h_{t+1}^j$ is not linked to $v_i$ through a path avoiding $w$.
		
		We prove in this case that $L=cl (V \setminus \{v_i , v_{i+1}\})$ is a good coline. As mentioned earlier, not all paths starting from $h_{t+1}^j$ go through $w$. It is easy to notice that if $h_{t+1}^j$ was linked to a sink lower than $v_i$ (a sink $v_t$ such that $t > i$) through a path avoiding $w$ then it is also be linked to $v_i$ through a path avoiding $w$. That is due to the planarity of the \reflectbox{L}-graph and the connection property. So here we are back to the previous possibility. Therefore $h_{t+1}^j$ must be linked to a sink higher than $v_i$ through a path $p$ avoiding $w$. This path $p$ blocks any source to the left of $h_{t+1}^j$ from reaching $v_i$ (because of the planarity property). This means that the only sources linked to $v_i$ are the sources between $v_{i+1}$ and $h_{t}^j$, and that due to the connection property, any source linked to $v_{i+1}$ is also linked to $v_i$. Therefore the set $H_1 = L \cup v_i$ is closed and of rank $r-1$. Similarly, the set $H_1 = L \cup v_{i+1}$ is closed and of rank $r-1$. Thus $H_1$ and $H_2$ are two simple copoints on $L$. Note that the set $L \cup h^{i+1}$ is not necessarily closed since it might contain other sources linked to $v_i$ but intersecting the path from $h^{i+1}$ to $v_{i+1}$. Figure \ref{case2.1.2} gives an example of this case.
		
		\begin{center}
			\begin{minipage}{.4\textwidth}
			\begin{tikzpicture}[scale=1]
				\draw[help lines,line width=.4pt,step=1] (-2,-1) grid (0,-3);
				\draw[help lines,line width=.4pt,step=1] (-2,1) grid (2,-1);
				
				\draw[line width=0.5mm] (2,1)--(2,-1)--(1,-1)--(1,-2)--(0,-2)--(0,-3)--(-2,-3);
				
				\draw[thick,->](1.5,0.5) -- (2,0.5);
				\draw[thick,->](1.5,-0.5) -- (1.5,0.5);
				\draw[thick,dashed](0.5,-0.5) -- (0.5,0.5);
				\draw[thick,->](1.5,-1) -- (1.5,-0.5);
				\draw[thick,->](1.5,-0.5) -- (2,-0.5);
				\draw[thick,->](0.5,-2) -- (0.5,-1.5);
				\draw[thick,->](0.5,-1.5) -- (1,-1.5);
				\draw[thick,->](0.5,-1.5) -- (0.5,-0.5);
				\draw[thick,->](0.5,-0.5) -- (1.5,-0.5);
				\draw[thick,->](-0.5,-3) -- (-0.5,-2.5);
				\draw[thick,->](-0.5,-2.5) -- (0,-2.5);
				\draw[thick,->](-0.5,-2.5) -- (-0.5,-1.5);
				\draw[thick,->](-0.5,-1.5) -- (0.5,-1.5);
				\draw[thick,->](-1.5,-3) -- (-1.5,-2.5);
				\draw[thick,->](-1.5,-2.5) -- (-0.5,-2.5);
				\draw[thick,->](-1.5,-2.5) -- (-1.5,-0.5);
				\draw[thick,->](-1.5,-0.5) -- (0.5,-0.5);
				\draw[thick,dashed](1.5,0.5) -- (1.5,1.5);
				\draw[thick,dashed](-1.5,-2.5) -- (-2.5,-2.5);
				\draw[thick,dashed](-1.5,-0.5) -- (-2.5,-0.5);
				\draw[thick,dashed](-1.5,-0.5) -- (-1.5,1);

				\draw(2,0.5) node[right] {$v_{i}$};
				\draw(2,-0.5) node[right] {$v_{i+1}$};
				\draw(1.5,-1) node[anchor=north west] {$h^{i+1}$};
				\draw(1,-1.5) node[anchor=north west] {$v_{j-1}$};
				\draw(0.5,-2) node[anchor=north west] {$h^{j-1}$};
				\draw(0,-2.5) node[anchor=north west] {$v_{j}$};
				\draw(-0.5,-3) node[below] {$h_{t}^j$};
				\draw(-1.5,-3) node[below] {$h_{t+1}^j$};
				\draw(-0.5,-2.5) node[anchor=south east] {$w$};
				
				\foreach \Point in {(-0.5,-3),(0.5,-2),(0,-2.5),(1.5,-1),(2,0.5),(1.5,0.5),(-1.5,-2.5),(-1.5,-0.5),(1,-1.5),(2,-0.5),(-1.5,-3),(0.5,-0.5),(1.5,-0.5),(0.5,-1.5),(-0.5,-1.5),(-0.5,-2.5)}{
					\node at \Point {\textbullet};}
			\end{tikzpicture} 
			\captionof{figure}{}
			\label{case2.1.1}
		\end{minipage}\hspace{1cm}
		\begin{minipage}{.4\textwidth}
			\begin{tikzpicture}[scale=1]
				\draw[help lines,line width=.4pt,step=1] (-2,-1) grid (0,-3);
				\draw[help lines,line width=.4pt,step=1] (-2,1) grid (2,-1);
				
				\draw[line width=0.5mm] (2,1)--(2,-1)--(1,-1)--(1,-2)--(0,-2)--(0,-3)--(-2,-3);
				
				\draw[thick,->](1.5,0.5) -- (2,0.5);
				\draw[thick,->](1.5,-0.5) -- (1.5,0.5);
				\draw[thick,dashed](0.5,-0.5) -- (0.5,0.5);
				\draw[thick,->](1.5,-1) -- (1.5,-0.5);
				\draw[thick,->](1.5,-0.5) -- (2,-0.5);
				\draw[thick,->](0.5,-2) -- (0.5,-1.5);
				\draw[thick,->](0.5,-1.5) -- (1,-1.5);
				\draw[thick,->](0.5,-1.5) -- (0.5,-0.5);
				\draw[thick,->](0.5,-0.5) -- (1.5,-0.5);
				\draw[thick,->](-0.5,-3) -- (-0.5,-2.5);
				\draw[thick,->](-0.5,-2.5) -- (0,-2.5);
				\draw[thick,->](-0.5,-2.5) -- (-0.5,-1.5);
				\draw[thick,->](-0.5,-1.5) -- (0.5,-1.5);
				\draw[thick,->](-1.5,-3) -- (-1.5,-2.5);
				\draw[thick,->](-1.5,-2.5) -- (-0.5,-2.5);
				\draw[thick](-1.5,-2.5) -- (-1.5,1.5);
				\draw[thick,dashed](1.5,0.5) -- (1.5,1.5);
				\draw[thick,dashed](-0.5,-1.5) -- (-0.5,-0.5);
				\draw[thick,dashed](-1.5,-2.5) -- (-2.5,-2.5);

				\draw(2,0.5) node[right] {$v_{i}$};
				\draw(2,-0.5) node[right] {$v_{i+1}$};
				\draw(1.5,-1) node[anchor=north west] {$h^{i+1}$};
				\draw(1,-1.5) node[anchor=north west] {$v_{j-1}$};
				\draw(0.5,-2) node[anchor=north west] {$h^{j-1}$};
				\draw(0,-2.5) node[anchor=north west] {$v_{j}$};
				\draw(-0.5,-3) node[below] {$h_{t}^j$};
				\draw(-1.5,-3) node[below] {$h_{t+1}^j$};
				\draw(-0.5,-2.5) node[anchor=south east] {$w$};
				
				\foreach \Point in {(-0.5,-3),(0.5,-2),(0,-2.5),(1.5,-1),(2,0.5),(1.5,0.5),(-1.5,-2.5),(1,-1.5),(2,-0.5),(-1.5,-3),(0.5,-0.5),(1.5,-0.5),(0.5,-1.5),(-0.5,-1.5),(-0.5,-2.5)}{
					\node at \Point {\textbullet};}
			\end{tikzpicture} 
			\captionof{figure}{}
			\label{case2.1.2}
		\end{minipage}
		\end{center}
	\end{itemize}
	
\item Subcase 2.2: There exists no two consecutive sources $h_{t}^j$ and $h_{t+1}^j$ such that $j > i$.

In this case, we prove that $L= cl (V \setminus \{v_i , v_r\})$ is a good coline. Since there exist no two consecutive sources $h_{t}^j$ and $h_{t+1}^j$ such that $j > i$, the connection property holds between $v_{i+1}$ and $h^r$. It follows that $h^r$ is linked to $v_i$, and $h^r$ is the last element in the positroid. It is not hard to check that $H_1 = L \cup v_r$ and $H_2 = L \cup h^r$ are both closed and of rank $r-1$. Therefore, $H_1$ and $H_2$ are two simple copoints on $L$. Figure \ref{case2.2} is an example of this case.

\begin{center}
\begin{tikzpicture}[scale=1]
	\draw[help lines,line width=.4pt,step=1] (-2,-1) grid (0,-3);
	\draw[help lines,line width=.4pt,step=1] (-2,1) grid (2,-1);
	\draw[help lines,line width=.4pt,step=1] (-2,-3)--(-2,-4);
	
	\draw[line width=0.5mm] (2,1)--(2,-1)--(1,-1)--(1,-2)--(0,-2)--(0,-3)--(-1,-3)--(-1,-4)--(-2,-4);
	
	\draw[thick,->](1.5,0.5) -- (2,0.5);
	\draw[thick,->](1.5,-0.5) -- (1.5,0.5);
	\draw[thick,dashed](0.5,-0.5) -- (0.5,0.5);
	\draw[thick,->](1.5,-1) -- (1.5,-0.5);
	\draw[thick,->](1.5,-0.5) -- (2,-0.5);
	\draw[thick,->](0.5,-2) -- (0.5,-1.5);
	\draw[thick,->](0.5,-1.5) -- (1,-1.5);
	\draw[thick,->](0.5,-1.5) -- (0.5,-0.5);
	\draw[thick,->](0.5,-0.5) -- (1.5,-0.5);
	\draw[thick,->](-0.5,-3) -- (-0.5,-2.5);
	\draw[thick,->](-0.5,-2.5) -- (0,-2.5);
	\draw[thick,->](-0.5,-2.5) -- (-0.5,-1.5);
	\draw[thick,->](-0.5,-1.5) -- (0.5,-1.5);
	\draw[thick,->](-1.5,-3.5) -- (-1.5,-2.5);
	\draw[thick,->](-1.5,-3.5) -- (-1,-3.5);
	\draw[thick,->](-1.5,-4) -- (-1.5,-3.5);
	\draw[thick,->](-1.5,-2.5) -- (-0.5,-2.5);
	\draw[thick,dashed](-1.5,-2.5) -- (-1.5,-1.5);
	\draw[thick,dashed](1.5,0.5) -- (1.5,1.5);
	\draw[thick,dashed](-0.5,-1.5) -- (-0.5,-0.5);

	\draw(2,0.5) node[right] {$v_{i}$};
	\draw(2,-0.5) node[right] {$v_{i+1}$};
	\draw(-1.5,-4) node[below] {$h^r$};
	\draw(-1,-3.5) node[right] {$v_r$};

	\foreach \Point in {(-0.5,-3),(-1,-3.5),(-1.5,-3.5),(0.5,-2),(0,-2.5),(1.5,-1),(2,0.5),(1.5,0.5),(-1.5,-2.5),(1,-1.5),(2,-0.5),(-1.5,-4),(0.5,-0.5),(1.5,-0.5),(0.5,-1.5),(-0.5,-1.5),(-0.5,-2.5)}{
		\node at \Point {\textbullet};}
\end{tikzpicture} 
\captionof{figure}{}
\label{case2.2}
\end{center}	
	
\end{itemize}
\end{proof}

We put this result in the context of the results from Goddyn et al.\cite{goddynetal} and recall the definition of a GSP- oriented matroid that was introduced there. We say that an oriented matorid $\mathcal{O}$ is a \emph{generalized series-parallel} (GSP) if every simple minor of $\mathcal{O}$ has a $\{0, \pm 1\}$-valued coflow which has
at most two nonzero entries. In the same work, the authors show that every GSP-oriented matroid with no loops has a NZ-3 coflow, i.e.\ is 3-colorable (see appendix). Additionally, Hochst\"{a}ttler and Guzman-Pro recently proved the following proposition.



\begin{prop}[Dual version of Proposition 16 \cite{santiago}] \label{santiago}
	Let $\mathcal{C}$ be a minor closed class of orientable matroids. If every simple matroid in $\mathcal{C}$ has a good coline, then every orientation of a matroid in $\mathcal{C}$ is GSP.
\end{prop}

Since positroids are closed under taking minors, Proposition \ref{santiago} along with Theorem \ref{goodcoline} imply that the class of simple positroids is a subclass of the class of GSP-oriented matroids. Thus, the main result follows immediately.




\begin{thm}\label{maintheorem}
	Simple positroids of rank at least $2$ are 3-colorable.
\end{thm}

\section{Conclusion}
Originally our plan was to prove that every simple gammoid of rank at least $2$ is 3-colorable by proving the existence of a positive coline, but as mentioned in the introduction, this was proven to be false by Guzman-Pro and Hochst\"{a}ttler in \cite{santiago}. Nevertheless, the authors showed in the same work that we don't necessarily need a positive coline to prove the 3-colorability of an oriented matroid, we only need a coline with at least two simple copoints, and in this way they proved that cobicircular matroids are 3-colorable. This leads to the following question which, if answered positively, would prove the 3-colorability of gammoids since gammoids are closed under minors.

\begin{que}
	Does every simple gammoid have a good coline?
\end{que}


\bibliographystyle{plain}
\bibliography{3color}

\appendix

\section*{Appendix}

As mentioned in the introduction, Hochst\"{a}ttler and Nickel introduced in \cite{goddynetal} the notion of coloring to general oriented matroid. We recall here their definition of the chromatic number of an oriented matroid. We assume basic familiarity with matroid theory and oriented matroids, and standars references are \cite{oxleybook} and \cite{ombibel}.

We start by recalling the definition of integer lattices.




\begin{defn}
	For a finite set of non-zero integer vectors $V:= \{v_1 , \dots , v_r\} \subset \mathbb{Z}^n$ let 
	\begin{equation*}
		\text{lat } V = \biggl\{ \sum_{i=1}^{r} \lambda_i v_i \; | \; \lambda_i \in \mathbb{Z} \biggr\}
	\end{equation*}
	denote the \emph{integer lattice generated by} V. For a family $\mathcal{X}$ of signed subsets we denote by lat $\mathcal{X} :=$ lat $\{ \vec{X} \; | \;  X \in \mathcal{X} \}  $ the integer lattice generated by its signed characteristic vectors. 
\end{defn}

The chromatic number $\chi(\mathcal{O})$ of an oriented matroid $\mathcal{O}$ can now be defined using the coflow lattice of the oriented matroid, defined below. 

\begin{defn}
	The \emph{coflow lattice} of an oriented matroid $\mathcal{O}$ on a finite set $E$, denoted as $\mathcal{F_{O^*}}$, is the integer lattice generated by the signed characteristic vectors of the signed cocircuits $\mathcal{D}$ of $\mathcal{O}$, where the \emph{signed characteristic vector} of a signed cocircuit $D=(D^+ , D^-)$ is defined by 
	\begin{equation*}
		\vec{D}(e) =  \begin{cases}
			1 \, \text{if}\, e \in D^+ \\
			-1 \, \text{if}\, e \in D^- \\
			0 \; \text{otherwise}
		\end{cases}
	\end{equation*}
	That is $\mathcal{F_{O^*}}= \biggl\{  \sum_{D \in \mathcal{D}} \lambda_D \vec{D} | \lambda_i \in \mathbb{Z} \biggr\} \subseteq \mathbb{Z}^{|E|}$
	
	We call any $x \in \mathcal{F_{O^*}}$ a coflow of $\mathcal{O}$. Such an $x$ is a \emph{nowhere-zero-k coflow (NZ-k-coflow)} if $0 < |x(e)| < k$ holds for $e \in E$. The \emph{chromatic number $\chi(\mathcal{O})$} is the minimum $k$ such that $\mathcal{O}$ has a NZ-k coflow.
\end{defn}

\end{document}